\documentclass[12pt,reqno]{amsart}
\usepackage[top=1in, bottom=1in, left=1in, right=1in]{geometry}
\usepackage{times, amsthm, amssymb, amsmath, amsfonts, bm, graphicx,
  mathrsfs, hyperref,wasysym} 
\usepackage[all]{xy}
\usepackage[usenames,dvipsnames]{color}

\newtheorem{theorem}{Theorem}[section]

\newtheorem{corollary}[theorem]{Corollary}
\newtheorem{proposition}[theorem]{Proposition}

\theoremstyle{definition}
\newtheorem{example}[theorem]{Example}
\newtheorem{remark}[theorem]{Remark}

\newtheorem{definition}[theorem]{Definition}
\newtheorem{convention}[theorem]{Convention}

\newtheorem{notation}[theorem]{Notation}

\DeclareMathAlphabet{\mathpzc}{OT1}{pzc}{m}{it}   

\def\<{\langle}
\def\>{\rangle}
\def\0{\mathbf{0}}

\def\CC{{\mathbb C}}

\def\calH{{\mathcal H}}

\def\calK{{\mathcal K}}

\def\NN{{\mathbb N}}

\def\PP{{\mathbb P}}
\def\QQ{{\mathbb Q}}

\def\RR{{\mathbb R}}

\def\ZZ{{\mathbb Z}}

\def\pp{{\mathfrak p}}

\def\ba{{\breve{a}}}
\def\bA{{\breve{A}}}
\def\bC{{\breve{C}}}

\def\del{\partial}

\def\boldone{\boldsymbol{1}}
\def\boldzero{\boldsymbol{0}}
\def\ini{\operatorname{in}}

\DeclareFontFamily{U}{mathx}{\hyphenchar\font45}
\DeclareFontShape{U}{mathx}{m}{n}{<-> mathx10}{}
\DeclareSymbolFont{mathx}{U}{mathx}{m}{n}
\DeclareMathAccent{\widebar}{0}{mathx}{"73}

\def\qdeg{{\rm qdeg}} 

\def\conv{{\rm conv}}
\def\rank{{\rm rank }}
\def\Var{{\rm Var}}
\def\gr{{\rm gr}}

\DeclareMathOperator{\SingLocus}{{\operatorname{Sing}}}
\def\Sing#1{{\SingLocus({#1})}}
\DeclareMathOperator{\dSingLocus}{{\operatorname{Sing^1}}}
\def\dSing#1{{\dSingLocus({#1})}}

\def\charVar{{\operatorname{Char}}}

\def\minus{\smallsetminus}

\def\endrk{\hfill$\hexagon$}

\numberwithin{equation}{section}
\begin{document}
\title{
Singularities and holonomicity of 
binomial \texorpdfstring{$D$}{D}-modules 
} 

\author{Christine Berkesch Zamaere}
\address{School of Mathematics, University of Minnesota, Minneapolis, MN 55455.}
\email{cberkesc@math.umn.edu}
\thanks{CBZ was partially supported by NSF Grants DMS 1303083, DMS 0901123, and OISE 0964985.}

\author{Laura Felicia Matusevich}
\address{Department of Mathematics \\
Texas A\&M University \\ College Station, TX 77843.}
\email{laura@math.tamu.edu}
\thanks{LFM was partially supported by NSF Grants DMS 0703866 and DMS
  1001763 and a Sloan Research Fellowship} 

\author{Uli Walther}
\address{Department of Mathematics\\ Purdue University\\ West Lafayette, IN \ 47907.}
\email{walther@math.purdue.edu}
\thanks{UW partially supported by NSF Grant DMS~0901123}


\begin{abstract}
We study binomial $D$-modules, which generalize $A$-hypergeo\-me\-tric
systems. We determine explicitly their singular loci and provide three
characterizations of their holonomicity.  The first of these
states that a binomial $D$-module is holonomic if and only if its
corresponding singular locus is proper.  The second characterization
is an equivalence of holonomicity and $L$-holonomicity for these
systems.  The third refines the second by giving more detailed
information about the $L$-characteristic variety of a non-holonomic
binomial $D$-module.
\end{abstract}
\maketitle

\setcounter{tocdepth}{1}
\section{Introduction}
\label{sec:intro}

\vspace*{-2mm}
The main object of study in this article is the class of  binomial $D$-modules, which generalize the $A$-hypergeometric system studied by Gelfand, Kapranov, and Zelevinsky~\cite{GGZ,GKZ,gkz irreducibility}. We recall the definitions of both objects in Definition~\ref{def:binomialD-mod}. 

Let $X$ be affine $n$-space over $\CC$, with coordinates $x_1,\dots,x_n$. The Weyl 
algebra $D$ is the ring of differential operators on $X$; it is generated by the
multiplication operators $x_1,\dots,x_n$ and the differentiation operators $\del_{1} := 
\frac{\del}{\del x_1},\dots,\del_{n} := \frac{\del}{\del x_n}$, subject to the Leibniz 
rule $\del_{j}x_i - x_i\del_{j}=\delta_{ij}$ (the Kronecker delta).

A \emph{projective weight vector on $D$} is
$L=(L_x,L_{\del})\in\QQ^{n}\times\QQ^n$ such that $L_x+L_{\del} =
c\cdot\boldone_n := c\cdot(1,\dots,1)$ for some constant $c>0$.  This
determines an increasing filtration $L$ on $D$ by $L^k D := \CC\cdot
\{x^u\del^v \mid L\cdot(u,v) \leq k\}$ for $k \in \QQ$.  Set $L^{<k} D
:= \bigcup_{\ell < k} L^{\ell} D$.  Since $c>0$, the associated graded
ring $\gr^L D$ is isomorphic to the coordinate ring of $T^*X\cong
\CC^{2n}$, which is a polynomial ring in $2n$ variables.  For any $P$
in $L^k D\minus L^{<k} D$, set $\ini_L(P):= P+L^{<k}D\in\gr^{L,k} D :=
L^k D / L^{<k} D \subseteq \gr^L D$ and $\deg^L(P) := k$.  By a slight
abuse of notation, set $x_i := \ini_L(x_i)$ and $\xi_i :=
\ini_L(\del_i)$, where $(x,\xi)$ are coordinates on $T^*X$.

For a left $D$-ideal $I$, set $\gr^L(I):= \<\ini_L(P)\mid P\in I\>\subseteq \gr^L(D)$. 
The \emph{$L$-characteristic variety} of the module $D/I$ is
\vspace{-1mm}
\begin{equation}
\label{eqn:charVar}
\charVar^L(D/I) := \Var(\gr^L(I)) \subseteq T^*X\cong\CC^{2n}.
\end{equation}
The projective weight vector $F = (\boldzero_n,\boldone_n) :=
(0,\dots,0,1,\dots,1)\in\QQ^{2n}$ induces the \emph{order filtration}
on $D$. The $F$-characteristic variety of a $D$-module is usually
called its characteristic variety. The \emph{singular locus} of $D/I$, denoted
$\Sing{D/I}$, 
is the projection of $\charVar^F(D/I)
\minus \Var(\xi_1,\dots,\xi_n)$ onto $X$, and as such, it is a closed subvariety of $X$. 

The \emph{divisorial singular locus} of $D/I$, denoted
by $\dSing{D/I}$, is the codimension at most one part of $\Sing{D/I}$.
From the point of view of holomorphic solutions
of systems of differential equations, 
there is no difference between $\Sing{D/I}$ and $\dSing{D/I}$ because
the codimension two  
singularities of holomorphic functions can be removed (see, for
instance,~\cite[Theorem~7.7]{scv}).

Our first main result generalizes to binomial $D$-modules
(Definition~\ref{def:binomialD-mod})  a result of
Gelfand and Gelfand that describes the divisorial singular loci of
$A$-hypergeometric systems with regular
singularities~\cite{gelfand-gelfand};  
Theorem~\ref{thm:singLocusA-hyp} contains their result as a special case. 
Theorem~\ref{thm:intro:binomialSing} will be later restated in more detail and proven as 
Corollary~\ref{cor:singLocus:holonomic:binomial}. 

\begin{theorem}
\label{thm:intro:binomialSing}
The divisorial singular locus of a holonomic binomial $D$-module is the
union of an explicit family of discriminantal varieties.  
\end{theorem}

For a left $D$-ideal $I$, 
the fact that $\dim(\charVar^F(D/I)) \geq n$ is known as Bernstein's inequality~\cite{bernstein} (see also~\cite{smith-characteristic-variety}). This result motivates one
of the key notions in the theory of $D$-modules: holonomicity. 

\begin{definition}
\label{def:holonomic}
The $D$-module $D/I$ is
\emph{$L$-holonomic} if its $L$-characteristic
variety~\eqref{eqn:charVar} is empty or has dimension $n$. We say
$D/I$ is \emph{holonomic} if it is $F$-holonomic 
for $F = (\boldzero_n,\boldone_n)$; the \emph{rank} of $D/I$ is
defined to be $\rank(D/I):= \dim_{\CC(x)}\CC(x)\otimes_{\CC[x]}D/I$.
\end{definition}

The next result of this article is a characterization of the
holonomicity of a binomial $D$-module in terms of its singular locus, which will be proven in Section~\ref{sec:finiteRank}.  

\begin{theorem} 
\label{thm:main:binomial:sing} 
A binomial $D$-module $M$ is holonomic if and only if its singular locus 
$\Sing{M}$ is a proper subset of $X$. 
\end{theorem} 

In Section~\ref{sec:holBinomialDmod}, we use the description from
Theorem~\ref{thm:intro:binomialSing} to provide another proof of Theorem~\ref{thm:main:binomial:sing} in the case that $\boldone_n$ is in the $\QQ$-rowspan of $A$. 

Our final results are two additional characterizations of the holonomicity of a binomial $D$-module in terms of $L$-holonomicity. 

\begin{theorem} 
\label{thm:main:binomial:2} 
Let $M$ be a binomial $D$-module. 
\vspace{-2mm}
\begin{enumerate}
\item \label{item:Lhol}
	The module $M$ is holonomic if and only if 
	$M$ is $L$-holonomic for some (equivalently, 
	every) projective weight vector $L$ on $D$.  
\item \label{item:LcharVar}
	Furthermore, the module $M$ is not holonomic 
	if and only 
	if $\charVar^L(M)$ has a component in 
	$T^*(\CC^*)^n$ of dimension greater than $n$ 
	for some (equivalently, every) projective 
	weight vector $L$ on $D$.
\end{enumerate}
\end{theorem}
 
We prove Theorem~\ref{thm:main:binomial:2} in Section~\ref{sec:Lhol}, using ideas
from~\cite{schulze-walther-duke}. 
Note that~\cite[Theorem~1.4.12]{SST} uses a Gr\"obner walk
argument to show equivalence of holonomicity and $L$-holonomicity for
any cyclic module $D/I$, but with different assumptions on $L$ than we make
here:~\cite{SST} requires that all coordinates of $L$ are nonnegative 
and $L_x+L_{\del}>0$ coordinatewise, while we ask for projective
weight vectors, whose coordinate sums are positive, but whose individual entries may be negative. 

\subsection*{Acknowledgements}
We thank Timur Sadykov and Bernd Sturmfels for insightful
conversations related to this work.  
We also thank an anonymous referee for pointing out a gap that
affected some of our previous proofs; in this version, we use new
arguments.  
A portion of this work was done during the Institut Mittag-Leffler
program on Algebraic Geometry with a view towards applications and the
MSRI program on Commutative Algebra; we are grateful to the program
organizers and participants for creating great research atmospheres. 

\section{\texorpdfstring{$A$}{A}-hypergeometric systems and binomial \texorpdfstring{$D$}{D}-modules}
\label{sec:binomial}

In this section, we introduce the $A$-hypergeometric systems and their 
generalizations, called binomial $D$-modules~\cite{DMM/D-mod}, and
we recall the structure of their $L$-characteristic varieties.

\begin{convention}
\label{conv:A}
Throughout this article, $A = [\,a_1\, a_2 \, \cdots\, a_n\,]$ is an integer 
$d\times n$ matrix whose columns span $\ZZ^d$ as a lattice. 
We also require that $A$ be \emph{pointed}, meaning that there exists
$h\in \QQ^d$ such that $h\cdot a_i >0$ for $i=1,\dots,n$. 
\end{convention}

The matrix $A$ determines a $(\CC^*)^d$-action on $X$ by 
\[
t \diamond p = (t^{a_1}p_1,\dots,t^{a_n}p_n) 
\ \text{ for } t=(t_1,\dots,t_d) \in
(\CC^*)^d \text{ and } p=(p_1,\dots,p_n) 
\in X.
\]
This action passes to the Weyl algebra $D$ via
\[
t \diamond x_i := t^{a_i}x_i \quad \text{ and } \quad t \diamond \del_{i} =
t^{-a_i}\del_{i} \quad \text{ for } i=1,\dots,n.
\]

\begin{definition}
\label{def:binomialD-mod}
Let $A=[a_{ij}]$ be as in Convention~\ref{conv:A}. We define the
\emph{Euler operators for A} to be
\begin{equation}
E_i := \sum_{j=1}^n a_{ij} x_j \del_{j} \quad \text{for }i=1,\dots,d.
\end{equation}
We denote by $E_A$ the sequence $E_1,\dots,E_d$, and for $\beta
\in\CC^d$, we denote by $E_A-\beta$ the sequence  $E_1-\beta_1,\dots,E_d-\beta_d$.
Let $I \subseteq \CC[\del_{1},\dots,\del_{n}]= \CC[\del]$ be a binomial
ideal, that is, an ideal generated by binomials and monomials. We
assume that $I$ is equivariant with respect to the $(\CC^*)^d$-action 
on $\CC[\del]$ induced by $A$.
Given $\beta \in \CC^d$, a \emph{binomial $D$-module} is of the form 
\[
\frac{D}{(I, E_A-\beta)}
:= \frac{D}{D \cdot I+D\cdot( E_A-\beta)}.
\]
\end{definition}

\begin{definition}
\label{def:A-hyp}
The binomial ideal
\begin{equation}
\label{eqn:toric}
I_A := \< \del^u - \del^v \mid  
u,v \in \NN^n,\, Au=Av\> \subseteq \CC[\del]
\end{equation}
is called the \emph{toric ideal} associated to $A$. Here we take the
convention that $\NN$ denotes the natural numbers including zero.  
The left $D$-ideal 
\[
H_A(\beta) := D \cdot (I_A, E_A-\beta)
\] 
is called an \emph{$A$-hyper\-geo\-met\-ric system}, and 
$D/H_A(\beta)$ is called an \emph{$A$-hyper\-geo\-met\-ric $D$-module}. 
\end{definition}

Every associated prime of a binomial ideal is also binomial, and every
prime binomial ideal is isomorphic to a toric ideal  
up to a rescaling of the variables~\cite{eisenbud-sturmfels}. 
Thus, toric ideals can be thought of as the building blocks for all binomial ideals. 
In~\cite{DMM/D-mod}, $A$-hypergeometric systems are shown to play a
similar fundamental role in the structure of binomial $D$-modules.
We now recall the characterization of holonomicity for binomial 
$D$-modules from~\cite{DMM/D-mod}. 

The action of the torus $(\CC^*)^d$ on $\CC[\del]$ defines a
multigrading on this ring, called the \emph{$A$-grading}, with
$\deg(x_i)=-\deg(\del_{i}) := a_i$. A binomial ideal $I\subseteq \CC[\del]$ is
torus equivariant if and only if it is $A$-graded. If
$M=\bigoplus_{\alpha \in \ZZ^d} M_\alpha$ is an $A$-graded
$\CC[\del]$-module, then the set of \emph{quasidegrees} of $M$ is
\[
\qdeg(M) := \overline{ \{ \alpha \in \ZZ^d \mid M_\alpha \neq 0 \}
}^{\rm Zariski} \subseteq \CC^d,
\]
where the closure is taken in the Zariski topology of $\CC^d$. 

Let $\mathscr{C}$ be a primary component of an $A$-graded binomial ideal 
$I \subseteq \CC[\del]$, which can be chosen to be binomial
by~\cite{eisenbud-sturmfels}. If the $A$-graded Hilbert function of
$\CC[\del]/\sqrt{\mathscr{C}\;}$ is bounded, then the component
$\mathscr{C}$, along with its corresponding associated prime 
$\sqrt{\mathscr{C}\;}$, is called \emph{toral}. Otherwise, they are
called \emph{Andean}. Examples and more details can be found 
in~\cite{DMM/primdec}. 

\begin{theorem}\cite[Theorem~6.3]{DMM/D-mod}
\label{thm:binomialHolonomicity}
Let $I\subseteq \CC[\del]$ be an $A$-graded binomial ideal. The binomial $D$-module 
$D/(I, E_A-\beta)$ is holonomic if and only if $-\beta$ lies outside the union of 
the sets $\qdeg(\CC[\del]/\mathscr{C})$, running over the Andean
components $\mathscr{C}$ of the binomial ideal $I$. 
\hfill$\square$
\end{theorem}

We next reproduce results of Schulze and
Walther~\cite{schulze-walther-duke} that set-theoretically describe 
the $L$-characteristic variety of $D/H_A(\beta)$.

Let $A$ and $h=(h_1,\dots,h_d)$ be as in Convention~\ref{conv:A}, and
let $L=(L_x,L_\del) \in \QQ^{2n}$ be a projective weight vector on $D$. 
Choose $\varepsilon>0$ such that $h\cdot a_i +\varepsilon\, L_{\del_i} >
0$ for $i=1,\dots,n$, and denote by $H_\varepsilon$ the hyperplane in $\PP^d_{\QQ}$
given by $\{ (y_0:y_1:\cdots:y_d) \in \PP^d_\QQ \mid \varepsilon \, y_0
+ h_1 y_1+\cdots +h_d y_d = 0\}$. 
The \emph{$L$-polyhedron of $A$} is the convex hull of $\{(1:\boldzero_d),
(L_{\del_1}:a_1),\dots,(L_{\del_n}:a_n)\}$ in the affine space $\PP^{d}_\QQ \minus H_\varepsilon$.
The \emph{$L$-umbrella of $A$}, denoted $\Phi(A,L)$, is the set of
faces of the $L$-polyhedron of $A$ that do not contain $(1:\boldzero_d)$. 

Let $\tau \in \Phi(A,L)$. We identify $\tau$ with 
the subset of $\{1,\dots,n\}$ indexing the columns 
of $A$ belonging to $\tau$. 
Whenever it is convenient, we also identify 
$\tau$ with the set $\{a_i\mid i \in \tau\}$, as well as with the
matrix whose columns are $a_i$ for $i\in \tau$. We denote by
$\bar{\tau}$ the set $\{1,\dots,n\} \minus \tau$.

Let $C_\tau$ denote the conormal space to the orbit under the torus
action of the point $\boldone_\tau$ in $\CC^n$ whose coordinates
indexed by $\tau$ are equal to $1$ and those indexed by $\bar{\tau}$
are equal to $0$.  Writing $x\xi := (x_1\xi_1,\dots,x_n\xi_n)$ and $\xi_\tau := \prod_{j \in \tau} \xi_j$, 
the Zariski closure of $C_\tau$, denoted $\overline{C_\tau}$, is equal to the Zariski closure in $T^*(\CC^n)$
of the variety in 
$T^*(\CC^n)\minus\Var(\xi_\tau)$ 
defined by 
\begin{align}
\label{eq:CtauClosure}
\CC[\xi_\tau^{-1}]\otimes_{\CC[\xi_\tau]} \left(
\<\xi_i\mid i \notin \tau\> + \<\xi^u-\xi^v \mid
u,v \in \NN^n , u_i=v_i =0 \text{ for } i \notin \tau, Au=Av \>+\<Ax\xi\>
\right). 
\end{align}
Note that the polynomials in~\eqref{eq:CtauClosure} can be viewed as
to not involve the variables $x_i$ for
$i\notin \tau$; in particular, if 
$(\overline{x},\overline{\xi})
\in \overline{C_\tau}$, then
$(\CC^{\bar{\tau}}\times 
\overline{x}_\tau) \times 
\{\overline{\xi}\}\subseteq \overline{C_\tau}$. 
Here $\CC^{\bar{\tau}}$ denotes the affine subspace of $X$ whose
coordinates indexed by $\tau$ are zero, and  
$\overline{x}_\tau$ is the point in $\CC^\tau$ whose coordinates
indexed by $\tau$ coincide with those of  
$\overline{x}$. 

\begin{theorem}\cite[see Corollary~4.17]{schulze-walther-duke}
\label{thm:A-hyp:charVar}
The $L$-characteristic variety of $D/H_A(\beta)$ is 
\[
\charVar^L(D/H_A(\beta)) 
= \bigsqcup_{\tau \in \Phi(A,L)} C_\tau 
= \bigcup_{\tau\in \Phi(A,L)} \overline{C_\tau}.
\]
\end{theorem}

We conclude this section with a result from~\cite{binomial-slopes} that describes the
$L$-characteristic variety of a holonomic binomial $D$-module. 

\begin{theorem}\cite[Theorem~4.3]{binomial-slopes}
\label{thm:binomial:charVar} 
\, 
If the binomial $D$-module $M = D/(I, E_A-\beta)$ is 
holonomic, then the $L$-characteristic variety of $M$ 
is the union of the $L$-characteristic varieties of 
the binomial modules $D/(\sqrt{\mathscr{C}\;}, E_A-\beta)$, 
where the union runs over the toral primary components $\mathscr{C}$
of $I$ such that $-\beta \in \qdeg(\CC[\del]/\mathscr{C})$.
\end{theorem}

\begin{remark}
\label{rmk:first part of item 2}
Note that Theorem~\ref{thm:binomial:charVar} in particular states that
if a binomial $D$-module is holonomic, then it is $L$-holonomic for
all projective weight vectors $L$. This follows because
$A$-hypergeometric $D$-modules, such as the modules
$D/(\sqrt{\mathscr{C}\;}, E_A-\beta)$ for $\mathscr{C}$ toral as
above, are always $L$-holonomic. This observation proves one of the
implications of Theorem~\ref{thm:main:binomial:sing}.
\endrk
\end{remark}

\section{Finite rank, singular loci, and holonomicity}
\label{sec:finiteRank}

In this section, we prove Theorem~\ref{thm:main:binomial:sing}, which relates holonomicity to the singular locus of a binomial $D$-module. 
We begin by pointing out that for an arbitrary left $D$-ideal $I$, the holonomicity of $D/I$
implies the properness of its 
singular locus~\cite[Subsection~5.4.6]{bjork}.  
That the converse holds is a special feature of binomial $D$-modules. 

\begin{example}
\label{example:non-holonomic:properSing}
Over $X=\CC^2$, the module $D/\<x_1\>$ is not holonomic because its
characteristic variety is a hypersurface in $\CC^4$. However,
$\Sing{D/\<x_1\>}= \Var(x_1)$, and as such, it is a proper subset of
$X$. For other examples of non-holonomic modules with proper singular loci, see Example~\ref{ex:nonholHorn-properSing} and~\cite[Example~1.4.10]{SST}. 
\endrk
\end{example}

As shown by Kashiwara, holonomicity implies finite rank (see~\cite[Proposition~1.4.9]{SST}).  
Note that although the module $D/\<x_1\>$ in
Example~\ref{example:non-holonomic:properSing} is not holonomic, it is
of finite rank. (In fact, it has rank $0$.)  
We show now that for an arbitrary cyclic module $D/I$, the finite rank
condition is equivalent to the properness of $\Sing{D/I}$. 

\begin{proposition}
\label{prop:mainSing:general}
If $I$ is a left $D$-ideal, then $D/I$ is of finite rank if and only if 
$\Sing{D/I}$ is a proper subset of $X$. 
\end{proposition}

\begin{proof}
To begin, note that by~\cite[Corollary 1.4.14]{SST}, the rank of $D/I$ 
equals 
\[ 
\dim_{\CC(x)} \CC(x)\otimes_{\CC[x]} \gr^F(D/I), 
\] 
which is the length of $\gr^F(D/I)_{\<\xi\>}$ over $\gr^F(D)_{\<\xi\>}$. 
Thus, $\rank(D/I)<\infty$ is equivalent to $\gr^F(D/I)_{\<\xi\>}$ being an Artinian 
$\gr^F(D)_{\<\xi\>}$-module. 
To put this another way, for all (minimal) primes $\pp$ of $\gr^F(I)$, the ring 
$(\gr^F(D)/\pp)_{\<\xi\>}$ is Artinian. Equivalently, for all (minimal) primes $\pp$ of 
$\gr^F(I)$, $\pp$ is not properly contained in $\<\xi\>$; in other words, such a $\pp$ either 
equals $\<\xi\>$ or contains an element in $\gr^F(D)$ that lies outside of $\<\xi\>$.
This means exactly that each (minimal) prime of $\gr^F(I)$ that is different from $\<\xi\>$ 
contains a nontrivial polynomial in $x$. This is because $\gr^F(I)$, and hence each of its 
(minimal) primes, is $\xi$-homogeneous and $\<\xi\>$ contains all elements of positive 
degree. Therefore, the rank of $D/I$ is finite exactly when
$\gr^F(I):\<\xi\>^\infty$ contains a polynomial in $x$, which is equivalent to 
properness of the singular locus of $D/I$, as desired.
\end{proof}

\begin{proof}[Proof of Theorem~\ref{thm:main:binomial:sing}]
It is known that for binomial $D$-modules, holonomicity is equivalent
to having finite rank~\cite[Theorem~6.3]{DMM/D-mod}.  Thus, the result
follows by combining~\cite[Proposition~1.4.9]{SST},
Proposition~\ref{prop:mainSing:general}, 
and~\cite[Theorem~6.3]{DMM/D-mod}.
\end{proof}

For any left $D$-ideal $I$, $\dSing{D/I}$ is a closed subvariety of $X$, and its properness is equivalent to the properness of $\Sing{D/I}$. 
Thus, in Section~\ref{sec:holBinomialDmod}, we use our description of the divisorial singular locus of a binomial $D$-module (in Corollary~\ref{cor:singLocus:holonomic:binomial} and Proposition~\ref{prop:singLocus=everything}) to 
give a more direct and elementary proof of Theorem~\ref{thm:main:binomial:sing} in the case that $\boldone_n$ is in the $\QQ$-rowspan of $A$. 

Theorem~\ref{thm:main:binomial:sing} 
was inspired by~\cite[Proof of Theorem~7]{PST}, where
Proposition~\ref{prop:mainSing:general}  
was used for Horn hyper\-geo\-metric systems. 
Horn hypergeometric systems are closely related to certain binomial
$D$-modules (see~\cite{bmw}); however, even for these systems, properness of the
singular locus is not equivalent to holonomicity, as
shown in the following example.

\begin{example}
\label{ex:nonholHorn-properSing}
On $X = \CC^3$, let $\theta_i := x_i\del_i$ for $1\leq i\leq 3$. 
The left $D$-ideal  
\begin{align*}
I = D\cdot\<
&(\theta_1+2\theta_2+\theta_3+2)\theta_1 - x_1(\theta_1+2\theta_2)\theta_1, \\
&
(\theta_1+2\theta_2+\theta_3+2)(\theta_1+2\theta_2+\theta_3+1)\theta_2
+ x_2(\theta_1+2\theta_2)(\theta_1+2\theta_2+1)\theta_2, \\ 
& (\theta_1+2\theta_2+\theta_3+2) + x_3\theta_3
\>
\end{align*}
is a nonconfluent Horn hypergeometric system 
of finite rank. 
However, $D/I$ is not holonomic, as witnessed by the component
$\Var(\<x_3,x_1\xi_1+x_2\xi_2\>)$ of its characteristic variety. 
On the other hand, computations in \texttt{Macaulay2}~\cite{M2} verify
that 
$\Sing{D/I}$ 
is indeed a proper subset of $X$.  
\endrk
\end{example}

\section{The singular locus of a binomial \texorpdfstring{$D$}{D}-module}
\label{sec:singLocus}

The goal of this section is to produce a polynomial that defines the
divisorial singular locus of a binomial $D$-module, see Corollary~\ref{cor:singLocus:holonomic:binomial}.  

We first consider the divisorial singular locus of an $A$-hypergeometric system. 
By definition, the divisorial singular locus of $D/H_A(\beta)$ is obtained by
removing the variety $\Var(\xi_1,\dots,\xi_n)$ from $\charVar^F(D/H_A(\beta))$, 
projecting the resulting set onto $X$, and discarding
the components of codimension two or higher.
Recall that $F = (\boldzero_n,\boldone_n)$ is the projective weight vector that
induces the order filtration on $D$. 
Using the description of $\charVar^F(D/H_A(\beta))$ as the union 
$\bigsqcup_{\tau \in \Phi(A,F)}C_\tau$ from Theorem~\ref{thm:A-hyp:charVar}, 
this projection procedure can be applied to each variety $C_\tau$ (or
$\overline{C_\tau}$) individually. 

A theorem of Gelfand, Kapranov, and Zelevinsky~\cite{gkz88,GKZ} 
describes the divisorial singular locus of an $A$-hypergeometric system when the
columns of $A$ lie on a hyperplane off the origin, equivalently, when
$D/H_A(\beta)$ is regular holonomic. 
The following Theorem~\ref{thm:singLocusA-hyp} generalizes this result by removing the assumption on $A$. 

Let $f = \bar{x}_1t^{a_1}+\cdots+\bar{x}_nt^{a_n}$. The Zariski closure of the set 
\[
\left\{
\bar{x} \in \CC^n \,\left|\, \exists\, \bar{t} \in (\CC^*)^d \text{ such that } 
f(\bar{t}) = \frac{\del f}{\del t_i}(\bar{t}) = 0 \text{ for } i=1,\dots, d
\right.
\right\} 
\]
is called the \emph{$A$-discriminantal variety}.
If this (irreducible) variety is a hypersurface, its defining polynomial is called
the \emph{$A$-discriminant}, denoted $\Delta_A$. 
If the codimension of the $A$-discriminantal variety is at least $2$,
then we set $\Delta_A=1$. The \emph{principal $A$-determinant}, denoted 
$\mathscr{E}_A$, is defined in~\cite[Chapter~10, Equation~1.1]{GKZ discrim}. 
For the purposes of this article, we need to know only that 
\[
\mathscr{E}_A = \pm \prod_{\tau \text{ face of } \conv(A)} (\Delta_\tau)^{\mu(\tau)}
\]
for certain positive integers $\mu(\tau)$. 
See \cite[Chapter~10, Theorem~1.2]{GKZ discrim} for more details, as
well as~\cite{horn-kapranov} for a parametric treatment of
$A$-discriminants. 

\begin{theorem}
\label{thm:singLocusA-hyp}
The divisorial singular locus of an $A$-hypergeometric system 
$D/H_A(\beta)$ is the zero set of the 
product of the 
$\tau$-discriminants, for the faces $\tau$ of $\Phi(A,F)$. 
In particular, if all of the columns of $A$ lie in a
hyperplane off the origin, then 
$\dSing{D/H_A(\beta)}= \Var(\mathscr{E}_A)$.
\end{theorem}

\begin{proof}
First consider the case that all of the columns of $A$ lie in a
hyperplane off the origin.  Then $\Phi(A,F)$ has a unique maximal face
$\tau=\{1,\dots,n\}$.  
Let $(\bar{x},\bar{\xi}) \in C_\tau$.  
Since $C_\tau$ is the conormal space to the orbit of $\boldone_\tau$, there is a 
$\bar{t} \in (\CC^*)^d$ such that 
$\bar{\xi}=(\bar{t}^{a_1},\dots,\bar{t}^{a_n})$. 
Consider
$f = \sum_{j=1}^n \bar{x}_j t^{a_j}$, a Laurent polynomial in 
$t_1,\dots,t_d$ with coefficients $\bar{x}_1,\dots,\bar{x}_n$. 
By our hypothesis on $A$, the vector $\boldone_n$ lies in the
$\QQ$-rowspan of $A$. Thus $A\overline{x\xi} = 0$
from~\eqref{eq:CtauClosure} imply 
that $f(\bar{t})=0$. In addition, 
\[
t_i\frac{\del f}{\del t_i} \bigg|_{t=\bar{t}} =
\sum_{j=1}^n (a_j)_i \bar{x}_j \bar{t}^{a_j} = \sum_{j=1}^n a_{ij}
\bar{x}_j \bar{\xi}_j = (A\overline{x\xi})_i = 0 
\quad\text{for } 1\leq i \leq d, 
\]
and therefore $\frac{\del f}{\del t_i}(\bar{t})=0$. 
Conversely, the vanishing of these derivatives at $\bar{t}$ implies
that $A\overline{x\xi} = 0$. 
We thus conclude that 
the codimension one part of the closure of the projection of $C_\tau$
onto the $x$-coordinates is the zero locus of the $A$-discriminant. 
Since we need to apply this procedure to every $\tau \in \Phi(A,F)$,
which in this case is the face lattice of
$\conv(A)=\conv(\{a_1,\dots,a_n\})$, we see that the divisorial singular locus of
$D/H_A(\beta)$ is the zero set of the product of the
$\tau$-discriminants for all the faces $\tau \in \Phi(A,F)$, which is
also the zero set of the principal $A$-determinant $\mathscr{E}_A$. 

For the general case, fix a maximal face $\tau$ of $\Phi(A,F)$. By
definition, the points $a_i \in \tau$ lie on a hyperplane off the 
origin. 
Thus using the argument above, we see that 
for all
$\sigma \in \Phi(A,F)$ with $\sigma \subseteq \tau$, projecting $C_\sigma$ to $X$ 
and discarding varieties of codimension two or higher contributes a
factor of the principal $\tau$-determinant $\mathscr{E}_\tau$ 
to the polynomial defining the divisorial singular locus of $D/H_A(\beta)$. 
In conclusion, $\dSing{D/H_A(\beta)}$ is the zero set of the product of
the $\tau$-discriminants, for all $\tau \in \Phi(A,F)$, or
equivalently, the zero set of the product of the principal
$\tau$-determinants, for all the maximal faces $\tau \in \Phi(A,F)$.
\end{proof}

Returning to the case of a binomial $D$-module $M = D/(I, E_A-\beta)$, 
recall from Theorem~\ref{thm:binomial:charVar} that the
$L$-characteristic variety of $M$ is the union of $L$-characteristic
varieties of $A$-hypergeometric
systems (up to a rescaling of the variables) given by certain toral associated primes of $I$.  
An alternative characterization for the toral 
components of $I$ is that $\mathscr{C}$ is toral 
if and only if
$D\cdot(\sqrt{\mathscr{C}\;}, E_A-\beta)$ 
is isomorphic to an $\tilde{A}$-hypergeometric 
system by rescaling the variables, where
$\tilde{A}$ is the submatrix of $A$ consisting of 
the columns $a_i$ such that 
$\del_{i} \notin \sqrt{\mathscr{C}\;}$~\cite[Corollary~4.8]{DMM/primdec}. 

This observation provides a description of the 
characteristic variety of a binomial 
$D$-module in terms of those of such
$\tilde{A}$-hypergeometric modules. 
Consequently, the divisorial singular locus of a holonomic 
binomial $D$-module can be expressed in
terms of principal $\tau$-determinants. 

\begin{corollary}
\label{cor:singLocus:holonomic:binomial}
The singular locus of a holonomic binomial $D$-module
$D/(I, E_A-\beta)$ is the union of 
$\Sing{D/(\sqrt{\mathscr{C}}, E_A-\beta)}$, 
where the union runs over the
toral primary components $\mathscr{C}$ of $I$ such that $-\beta \in
\qdeg(\CC[\del]/\mathscr{C})$.
The divisorial singular locus of $D/(I, E_A-\beta)$ is a union of rescaled discriminantal varieties,
given by the product of the  
rescaled $\tau$-discriminants for $\tau \in \Phi(A_\pp,L)$,
where, for a toral associated prime $\pp$ of $I$, the matrix $A_\pp$
has columns $a_i$ whenever $\del_i \notin \pp$.
\end{corollary}

\section{Holonomicity of binomial \texorpdfstring{$D$}{D}-modules via singular loci}
\label{sec:holBinomialDmod}

In this section, we give an alternative proof to Theorem~\ref{thm:main:binomial:sing} in the special case that $\boldone_n$ belongs to the $\QQ$-rowspan of $A$. 
We make this assumption throughout this section only. 

\begin{proposition}
\label{prop:singLocus=everything}
Let $A$ be as in Convention~\ref{conv:A} and assume that $\boldone_n$ is in the $\QQ$-rowspan of $A$.
Let $\bA \in \ZZ^{k\times n}$ be of full rank $k$ with $d<k<n$,
and assume that $A$ is the submatrix of $\bA$ consisting of its first $d$
rows.
Then 
$\dSing{D/(I_\bA, E_A-\beta)} = X$. 
\end{proposition}

In Proposition~\ref{prop:singLocus=everything}, 
the Euler operators in
$D/(I_\bA, E_A-\beta)$ 
come from $A$, so we are working with an 
$\bA$-hypergeometric system where some Euler 
operators have been removed. 
By a theorem of Kashiwara, 
the solution space of germs of holomorphic 
solutions at a nonsingular point of a holonomic 
module is finite dimensional~\cite[Theorem~1.4.19]{SST}. Thus $D/(I_\bA, E_A-\beta)$
is not holonomic because
its solution space has uncountable dimension. 
(It is the direct sum of the solution spaces of 
$D/(I_{\bA}, E_\bA -\breve{\beta})$, 
running over $\breve{\beta} \in \CC^k$ whose 
first $d$ coordinates coincide with $\beta$.) 

\begin{proof}[Proof of Proposition~\ref{prop:singLocus=everything}] 
Since
$D \cdot ( I_A , E_A-\beta) \supseteq 
	D \cdot (I_{\bA} , E_A-\beta)$,
the corresponding singular loci satisfy the reverse inclusion, namely
\[
\Sing{D /(I_A, E_A-\beta) } \subseteq
\Sing{D /(I_{\bA}, E_A-\beta)}.
\]
The singular locus on the right hand side is equivariant with respect to the action of $(\CC^*)^k$ induced by $\bA$, and in
particular, by the action of $(\CC^*)^{k-d}$ induced by $\bA \minus
A$, which is notation for the matrix whose rows are the last $k-d$
rows of $\bA$. Consequently, 
\begin{equation}
\label{eq:containmentOfLoci}
(\CC^*)^{k-d} \diamond 
	\Sing{D/(I_A, E_A-\beta)} 
	\subseteq
	\Sing{D/(I_{\bA}, E_A-\beta)}.
\end{equation}
Our goal is to prove that the set on the left hand side
of~\eqref{eq:containmentOfLoci} is dense in $X$, from which the result
follows, since singular loci are closed in $X$. 

The divisorial singular locus of $D/(I_A, E_A-\beta)$ is defined by the
principal ${A}$-determinant $\mathscr{E}_{A}$. The Newton polytope of
$\mathscr{E}_{A}$ is the \emph{secondary polytope of ${A}$}~\cite[Theorem~10.1.4]{GKZ discrim} has dimension
$m=n-d$~\cite[Theorem~7.1.7]{GKZ discrim}, and its real affine span 
is a translate of $\ker_{\RR}(A)$~\cite[Proposition~7.1.11]{GKZ discrim}.
This implies that $\mathscr{E}_{A}$ cannot be homogeneous with respect
to the grading induced by $\bA$, because if it were, the affine span
of its Newton polytope would be contained in a translate of
$\ker_{\RR}(\bA)$, which has dimension $n-k<n-d$.
Therefore $\mathscr{E}_A$ has an irreducible factor $G$ that is not
$\bA$-homogeneous, so $\Var(G)$ is not
$(\CC^*)^{k-d}$-equivariant. Note that $G$ cannot be a monomial, so a
generic point in the hypersurface $\Var(G)$ is not on any coordinate
hyperplane.  

Consider the map
\[
\begin{array}{rcl}
(\CC^*)^{k-d} \times 
\left(\Sing{D/(I_{A}, E_A-\beta) }\right) 
&\longrightarrow & 
\Sing{D/(I_{\bA}, E_A-\beta)} 
\subseteq X \\
(s,x) & \mapsto & s \diamond x
\vspace{-1mm}
\end{array}
\]
given by the action of $\bA \minus A$. The closure of the image of
$(\CC^*)^{k-d} \times \Var(G)$ under this map is an irreducible
subvariety of $X$. It contains the hypersurface $\Var(G)$, and this
containment is proper because $\Var(G)$ is not
$(\CC^*)^{k-d}$-equivariant. Therefore the closure of the image of
$(\CC^*)^{k-d} \times \Var(F)$ must be all of $X$. But then it follows
that $\Sing{D /(I_{\bA}, E_A-\beta)}=X$. 
\end{proof}

\begin{proof}[A second proof of Theorem~\ref{thm:main:binomial:sing}, in the case that the $\QQ$-rowspan of $A$ contains $\boldone_n$]
For $A$ as in Convention~\ref{conv:A} with $\boldone_n$ in the
$\QQ$-rowspan of $A$, let $I \subseteq \CC[\del]$ be an $A$-graded
binomial ideal. We consider the binomial $D$-module $M = D/(I,
E_A-\beta)$.

In light of the description of the divisorial singular locus of a holonomic
binomial $D$-module in Corollary~\ref{cor:singLocus:holonomic:binomial},
we may assume that $M$ is not holonomic.
All of the associated primes of $I$ are of the form 
$\mathfrak{p} = \CC[\del]\cdot
	(I_0+\<\del_{i}\mid i\notin\sigma\>)$, 
where $\sigma \subseteq \{1,\dots,n\}$, $I_0$ is generated by binomials in 
$\CC[\del_{j} \mid j \in \sigma]=:\CC[\del_\sigma]$, and 
$I_0 \cap \CC[\del_\sigma]$ is isomorphic to a toric ideal after
rescaling the variables~\cite[Corollary~2.6]{eisenbud-sturmfels}. 

By~\cite[Theorem~6.3]{DMM/D-mod}, if $M$ is not holonomic, then there
exists a primary component $\mathscr{C}$ of $I$ corresponding to an Andean 
associated prime $\mathfrak{p}$ such that $-\beta \in \qdeg(\CC[\del]/\mathscr{C})$. 
Theorem~5.6 from~\cite{DMM/D-mod} implies that $D/(\mathscr{C}, E_A-\beta)$ 
is not holonomic. The argument in the proof of~\cite[Theorem~5.6]{DMM/D-mod} allows us
to reduce to the case when $-\beta \in \qdeg(\CC[\del]/\mathfrak{p})$. 
In this case, the Andean condition ensures that $D/(\mathfrak{p}, E_A-\beta)$ 
is (after rescaling of the variables) a binomial $D$-module as in 
Proposition~\ref{prop:singLocus=everything}. 
From that result, we conclude that
\[
X = \Sing{D/(\mathfrak{p}, E_A-\beta)} 
\subseteq \Sing{D/(\mathscr{C}, E_A-\beta)} 
\subseteq \Sing{M}
\subseteq X,
\]
so that $\Sing{M}=X$, as desired.
\end{proof}

\section{The \texorpdfstring{$L$}{L}-holonomicity of binomial \texorpdfstring{$D$}{D}-modules} 
\label{sec:Lhol}

In this section, we prove Theorem~\ref{thm:main:binomial:2} using
ideas from~\cite{schulze-walther-duke}.

\begin{notation}
\label{not:bA}
For $A$ as in Convention~\ref{conv:A}, let $\bA \in \ZZ^{k\times n}$
denote a matrix of full rank $k$ with $d<k<n$, and assume that $A$ is
the submatrix of $\bA$ consisting of its first $d$ rows. Let
$L\in\QQ^n\times\QQ^n$ be a projective weight vector.
\end{notation}

Given a face $\tau\in\Phi(\bA,L)$, we 
let $\del_\tau = \prod_{i\in \tau}\del_i$ and 
use the following notation: 
\[
\bC_\tau := 
\{(x,\xi)\mid \xi_i\neq 0\text{ for all } i\in\tau\}
\cap 
\Var(\<\xi_i\mid i \notin \tau\> + \gr^L(I_\bA)+\gr^L(E_A-\beta)). 
\]
We will below modify some ideas in the proof of 
\cite[Theorem~3.10]{schulze-walther-duke} to suit our purposes. For
convenience, we list some of the key points here.  
Recall that the $L$-graded ideal 
$\gr^L(I_\bA) \subseteq \gr^L(D)$ has a minimal component for
every $\tau\in\Phi(\bA,L)$ of (highest possible) dimension $k-1$. 
Such a $\tau$ is called a \emph{facet} of $\Phi(\bA,L)$.
Since $I_\bA \subseteq \CC[\del]$, we have abused notation and written
$\gr^L(I_\bA)$ in place of $\gr^L(D \cdot I_\bA)$.

\begin{proposition}
\label{prop:Lchar-less-Eulers}
Let $\bA$ be as in Notation~\ref{not:bA}. Then 
for any facet $\tau\in\Phi(\bA,L)$, 
$\bC_\tau$ is contained in 
$\charVar^L(D/(I_\bA, E_A-\beta))$.
\end{proposition}
\begin{proof}
Fix a facet $\tau\in\Phi(\bA,L)$; then the matrix $\tau$ has full rank $k$. 
We may use the argument of~\cite[Theorem~3.10]{schulze-walther-duke},
which is based on the fact that $\gr^L(E_\bA)$ is a regular sequence on $\gr^L(D[\del_\tau^{-1}]/I_\bA)$. Thus $\gr^L(E_A)$ must be a regular sequence on $\gr^L(D[\del_\tau^{-1}]/I_\bA)$. 
A standard argument on the spectral sequence of a filtered complex shows that 
\[
\gr^L\left(\frac{D[\del_\tau^{-1}]}{(I_\bA,
  E_A-\beta)}\right) 
\ = \  
\frac{\gr^L (D[\del_\tau^{-1}])}{\gr^L(I_\bA)+\gr^L(E_A-\beta)}\,. 
\] 
(Compare, for example, with~\cite[Theorem~4.3.5]{SST} for a version close to our context.) 
Now note that 
\[
\gr^L(I_\bA)+\gr^L(E_A-\beta)
\subseteq\<\xi_i\mid i \notin \tau\> + \gr^L(I_\bA)+\gr^L(E_A-\beta), 
\]
so after localizing at $\del_\tau$, we conclude that 
$\bC_\tau$ is contained in $\charVar^L(D/(I_\bA, E_A-\beta))$.  
\end{proof}

\begin{proposition}
\label{prop:Lhol-viaCtau}
If $\bA$ is as in Notation~\ref{not:bA}, then 
$D/(I_\bA, E_A-\beta)$ fails to be $L$-holonomic for all projective weight vectors $L$.
\end{proposition}
\begin{proof}
Fix a projective weight vector $L$. 
By Proposition~\ref{prop:Lchar-less-Eulers}, $\bC_\tau\subseteq\charVar^L(D/(I_\bA, E_A-\beta))$ for each face $\tau\in\Phi(\bA,L)$ of dimension $k-1$. The proof also shows that $\gr^L(E_A-\beta)$ is a regular sequence on $\gr^L(D[\del_\tau^{-1}]/I_\bA)$, and therefore $\dim(\bC_\tau)=n+k-d>n$. Thus $D/(I_\bA, E_A-\beta)$ fails to be $L$-holonomic. 
\end{proof}

\begin{proof}[Proof of Theorem~\ref{thm:main:binomial:2}.\eqref{item:Lhol}]
Let $M=D/(I,E-\beta)$ be a binomial $D$-module. If $M$ is holonomic, then it is
$L$-holonomic for all projective weight vectors $L$, as we saw in
Remark~\ref{rmk:first part of item 2}. 

Now we assume that $M$ is a non-holonomic binomial $D$-module and
show that it is not $L$-ho\-lo\-no\-mic for any projective weight vector $L$.
The associated primes of $I$ are of the form 
$\mathfrak{p} = \CC[\del]\cdot
	(I_0+\<\del_{i}\mid i\notin\sigma\>)$, 
where $\sigma \subseteq \{1,\dots,n\}$, $I_0$ is generated by binomials in 
$\CC[\del_{j} \mid j \in \sigma]=:\CC[\del_\sigma]$, and 
$I_0 \cap \CC[\del_\sigma]$ is isomorphic to a toric ideal after
rescaling the variables~\cite[Corollary~2.6]{eisenbud-sturmfels}. 

By~\cite[Theorem~6.3]{DMM/D-mod}, if $M$ is not holonomic, then there
exists a primary component $\mathscr{C}$ of $I$ corresponding to an Andean 
associated prime $\mathfrak{p}$ such that $-\beta \in \qdeg(\CC[\del]/\mathscr{C})$. 
Theorem~5.6 from~\cite{DMM/D-mod} implies that $D/(\mathscr{C}, E_A-\beta)$ 
is not holonomic. The argument in the proof of~\cite[Theorem~5.6]{DMM/D-mod} allows us
to reduce to the case when $-\beta \in \qdeg(\CC[\del]/\mathfrak{p})$. 
In this case, the Andean condition ensures that $D/(\mathfrak{p}, E_A-\beta)$ 
is (after rescaling of the variables) a binomial $D$-module as in 
Notation~\ref{not:bA}. 
In particular, we have reduced to the case that $M = D/(I_\bA, E_A-\beta)$, where $\bA$ is as in Notation~\ref{not:bA}. 
Thus, the proof is complete by Proposition~\ref{prop:Lhol-viaCtau}. 
\end{proof}

We now prove a stronger version of
Proposition~\ref{prop:Lhol-viaCtau}. 

\begin{proposition}
\label{prop:Lhol-GKZ-viaOrderFiltration}
If $\bA$ is as in Notation~\ref{not:bA}, then 
$\charVar^L(D/(I_\bA, E_A-\beta))$ 
has a component in $T^*(\CC^*)^n$ of dimension $n+k-d$. 
\end{proposition}
\begin{proof}

Let $\ba_i$ denote the $i$th column of $\bA$, 
and consider the $\bA$-grading on $D$, given by $\deg(x_i) = \ba_i = -\deg(\del_i)$.
Fix some $\beta_\bA\in\CC^k$ that agrees with $\beta$ in its first $d$
coordinates, and let $\calK_\bullet(E_\bA-\beta_\bA;N)$ and
$\calH_\bullet(E_\bA-\beta_\bA;N)$ respectively denote the Euler--Koszul
complex and its homology, where $N$ is an $\bA$-graded
$\CC[\del]$-module and the action of $E_i-\beta_i$ on an
$\bA$-homogeneous element $y$ in $D\otimes_{\CC[\del]} N$ is given by  
$(E_i - \beta_i) \circ y := (E_i - \beta_i - \deg_i(y))y$.  
Background and properties of Euler--Koszul homology can be found in~\cite{MMW,schulze-walther-duke}. 

Recall that the $L$-initial terms of $x_j$ and $\del_j$ are denoted by
$x_j$ and $\xi_j$ respectively. 
Let $\tau$ be a facet of $\Phi(\bA,L)$, and
set $\xi_\tau := \prod_{j \in \tau} \xi_j$ and $\del_\tau :=
\prod_{j\in \tau} \del_j$.
By~\cite{schulze-walther-duke}, 
the spectral sequence
\[
H_\bullet(\gr^L(D[\del_\tau^{-1}]\otimes_D
\calK_\bullet(E_\bA-\beta_{\bA};D/I_\bA)))\Rightarrow
\gr^L(D)[\xi_\tau^{-1}]\otimes_{\gr^L(D)}
\gr^L(\calH_\bullet(E_\bA-\beta_\bA,D/I_\bA))
\] 
induced by the $L$-filtration on the localized Euler--Koszul complex
to $D/I_\bA$ collapses.
We recall the reason for the collapse.
Localizing $\gr^L(D)/\gr^L(I_\bA)$ at $\xi_\tau\not=0$ singles out the
component of $\gr^L(I_\bA)$ that is primary to $\gr^L(I_\tau) +\<\xi_i
\mid i \notin \tau\>$, where
$I_\tau \subseteq \CC[\del_j \mid j \in \tau]$ is the toric ideal defined by (the
submatrix of $\bA$ whose columns are indexed by) $\tau$. 
In all smooth points of the zero set of $\gr^L(I_\tau) +
\<\xi_i \mid i \notin \tau\>$, 
the symbols $\gr^L(E_\bA)$ cut out the conormal space $C_\tau$ of
the  $(\CC^*)^k$-orbit under the action
 defined by $\bA$ on the point
$\boldone_\tau$. The collapse then follows from 
$\gr^L(E_{\bA}-\beta_\bA)$ forming a
regular sequence on 
$\gr^L(D)[\xi_\tau^{-1}]/ \left(\gr^L(D) [\xi_\tau^{-1}]\cdot 
\gr^L(I_{\bA}) \right)$.

The facet $\tau$ may or may not be a pyramid in the sense
of~\cite[Definition 2.4]{SW12}. (It is said that $\tau$ is a
\emph{pyramid} if there is a column $i\in\tau$ with $\dim(\tau \minus i) < \dim(\tau )$.) 
It follows from Remark~2.5 in
loc.~cit.\ that there is a unique face $\sigma$ of $\tau$ such that
$\tau$ is a pyramid over $\sigma$ and $\sigma$ is not a pyramid. In
particular, slightly abusing notation, $I_\tau=I_\sigma$ and $I_\tau$-primary
ideals are $I_\sigma$-primary ideals. Hence 
\[
H_\bullet(\gr^L(D[\del_\sigma^{-1}]\otimes_D
\calK_\bullet(E_\bA-\beta_{\bA};D/I_\bA)))\Rightarrow
\gr^L(D)[\xi_\sigma^{-1}]\otimes_{\gr^L(D)}
\gr^L(\calH_\bullet(E_\bA-\beta_\bA,D/I_\bA))
\]
also collapses because $\gr^L(E_{\bA})$ is a regular sequence on
$\gr^L(D)[\xi_\sigma^{-1}]/\left(\gr^L(D)[\xi_\sigma^{-1}]\cdot
\gr^L(I_{\bA})\right)$. 
In particular, 
$\gr^L(D)[\xi_\sigma^{-1}]\otimes_{\gr^L(D)}
\gr^L(D/H_\bA(\beta_\bA))=\gr^L(D)[\xi_\sigma^{-1}]/\left(\gr^L(D)[\xi_\sigma^{-1}]
\cdot(P_\tau+\gr^L(E_\bA))\right)$, 
where $P_\tau$ is the $\left(\gr^L(I_\tau)+\<\xi_i \mid i\notin
\tau\>\right)$-primary component of $\gr^L(I_\bA)$. 

It follows that in $\gr^L(D)[\xi_\sigma^{-1}]$, $\gr^L(E_A)$ is regular on $\gr^L(I_\bA)$
and 
\[
\gr^L(D)[\xi_\sigma^{-1}]\otimes_{\gr^L(D)} \gr^L(D/(I_\bA, E_A-\beta)=
\gr^L(D)[\xi_\sigma^{-1}]/\left(\gr^L(D)[\xi_\sigma^{-1}]\cdot(P_\tau+\gr^L(E_A))\right).
\] 
We now must show that the variety of  
$(P_\tau+\gr^L(E_A))\subseteq \gr^L(D)[\xi_\sigma^{-1}]$ has a component of
dimension at least $n+1$ that meets $T^*(\CC^*)^n$.  

The regularity of $\gr^L(E_A)$ on $\gr^L(I_\bA)$ in the localized ring
ensures that every component of
$\gr^L[\xi_\sigma^{-1}](P_\tau+\gr^L(E_A))$ has dimension
$2n-d-(n-d)=n+k-d$, as $\tau$ is a facet and so the height of $P_\tau$
is $n-d$. It
remains to be shown that one of its components meets $T^*(\CC^*)^n$. 
We may safely replace $P_\tau$ by 
$\sqrt{P_\tau\,} = 
\gr^L(I_{\tau})+\<\xi_i \mid i \notin \tau\> 
= 
\gr^L(I_\sigma) + \<\xi_i \mid i \notin\tau\>$. Then we have
\[
\gr^L(D)[\xi_\sigma^{-1}]
\big(
\gr^L(I_\sigma)+\<\xi_i \mid
i \notin \tau\> +\gr^L(E_A)
\big)
=
\gr^L(D)[\xi_\sigma^{-1}]
\big(
\gr^L(I_\sigma)+\<\xi_i
\mid i \notin \tau\>+\gr^L(E_{A,\tau})
\big),
\] 
where $E_{A,\tau}$ denotes the Euler operators
arising from the submatrix of $A$ whose columns are indexed by $\tau$.
As $\tau\supseteq\sigma$ is a pyramid, the right hand side above is equal to
\[
\gr^L(D)[{\xi_\sigma}^{-1}]\left(
\gr^L(I_\sigma) 
+ \<\xi_i \mid i \notin \tau\> 
+ \gr^L(E_{A,\sigma}) + 
\<x_i\xi_i \mid {i\in\tau\smallsetminus\sigma}\>
\right), 
\]
and, in particular, it is contained in the ideal
$\gr^L(D)[\xi_\sigma^{-1}]\left(\gr^L(I_\sigma) +\<\xi_i \mid i \notin
\sigma\>+\gr^L(E_{A,\sigma})\right)$ of height $n-d+k$. 

To avoid confusion, for $\sigma\in\Phi(\bA,L)$, we denote by $C_{\sigma,\bA}$ the conormal space to the orbit of the point $\boldone_\sigma\in\CC^n$ under the $(\CC^*)^k$-action (defined by $\bA$). Previously defined for $\tau\in\Phi(A,L)$ and the corresponding $(\CC^*)^d$-action, such a conormal was written $C_\tau$, see~\eqref{eq:CtauClosure}.

Suppose now that
$\Var\left(\gr^L(D)[\xi_\sigma^{-1}]\left(\gr^L(I_\sigma) + \<\xi_i
\mid i \notin \sigma\>+\gr^L(E_{A,\sigma})\right)\right)$ does not
meet $T^*(\CC^*)^n$. 
Note that this variety does, however, contain the
conormal $C_{\sigma,\bA}$. 
Hence none of its components containing $C_{\sigma,\bA}$ can meet $T^*(\CC^*)^n$ in $\xi_\sigma\not=0$. 
Since components are by definition irreducible, any such component will be
contained in a hyperplane $\Var(x_i)$ inside the cotangent space of
$\CC^n\smallsetminus \Var(\xi_\sigma)$.  In particular, this must then
hold for $C_{\sigma,\bA}$ itself because its generic point is over
$\xi_\sigma\not=0$. We show now that this is impossible.

Let $\Var(x_i)$ be the presumed hyperplane that contains $C_{\sigma,\bA}$;
since for the $(\CC^*)^k$-orbit of $\boldone_\sigma$ 
the variable $x_i$ is a cotangent variable, this implies that the
toric ideal $I_\sigma$ does not involve $\del_i$. In turn, $\sigma$
must be a pyramid by~\cite[Remark~2.5]{SW12}. However, this
contradicts our choices, so $C_{\sigma,\bA}$ must meet $T^*(\CC^*)^n$.  
\emph{A fortiori}, so will any component of
$\gr^L(D)[\xi_\sigma^{-1}](P_\tau+\gr^L(E_A))$ that contains $C_{\sigma,\bA}$, and
hence $\gr^L(I_\bA + \<E_A-\beta\>)$ has a component of dimension $n+k-d$ that
meets $T^*(\CC^*)^n$.
\end{proof}

It would be interesting to determine the structure of
$\charVar^L(D/(I_\bA, E_A-\beta))$. 
Frequently, different facets of the $L$-umbrella of $A$ give
rise to the same ``bad" 
component in the proof of
Proposition~\ref{prop:Lhol-GKZ-viaOrderFiltration}. 

\begin{proof}[Proof of Theorem~\ref{thm:main:binomial:2}.\eqref{item:LcharVar}]
Fix a projective weight vector $L$ on $D$. 
By Theorem~\ref{thm:main:binomial:2}.\eqref{item:Lhol}, we may replace
``holonomic" by ``$L$-holonomic." 
The ``if'' direction is clear from the definition of
$L$-charac\-ter\-is\-tic variety. For the ``only if" direction, as in
the proof of Theorem~\ref{thm:main:binomial:2}.\eqref{item:Lhol}
above, we can immediate reduce to the case of $M$ as in
Proposition~\ref{prop:Lhol-GKZ-viaOrderFiltration}, establishing the
desired result. 
\end{proof}

\raggedbottom
\def\cprime{$'$} \def\cprime{$'$}
\providecommand{\href}[2]{#2}
\end{document}